%% file: ms.tex
\documentclass[11pt,a4paper,reqno]{amsart}
\pdfoutput=1

\input{setup}
\input{macros}

\title[Random walks on moduli space]{On convergence of random walks on moduli space}

\author{Roland Prohaska}
\address{Departement Mathematik, ETH Z\"{u}rich, R\"{a}mistrasse 101, 8092 Z\"{u}rich, Switzerland}
\email{roland.prohaska@math.ethz.ch}

\date{\usdate\today}

\begin{document}

\input{abstract}
\maketitle
\input{introduction}
\input{spectral_gap}
\input{pathwise}

\bibliographystyle{plain}
\bibliography{refs}

\end{document}

%% file: setup.tex
\usepackage[T1]{fontenc}
\usepackage[utf8]{inputenc}
\usepackage{lmodern}
\usepackage[english]{babel}

\usepackage{csquotes}

\usepackage{datetime}

\usepackage[dvipsnames]{xcolor}

\usepackage{amssymb}
\usepackage{amscd}
\usepackage{amsfonts}
\usepackage{mathtools}
\usepackage{dsfont}

\usepackage[activate]{pdfcprot}
\usepackage[pdftex,colorlinks,linkcolor=black,citecolor=black,filecolor=black]{hyperref}

\numberwithin{equation}{section}
\theoremstyle{plain}
\newtheorem{theorem}{Theorem}[section]
\newtheorem{proposition}[theorem]{Proposition}

\newtheorem{corollary}[theorem]{Corollary}
\theoremstyle{remark}

\theoremstyle{definition}
\newtheorem{definition}[theorem]{Definition}

\newsavebox{\proofbox}
\savebox{\proofbox}{\begin{picture}(7,7)  \put(0,0){\framebox(7,7){}}\end{picture}}

%% file: macros.tex
\let\parsign\S

\renewcommand\H{\mathcal{H}}

\newcommand\M{\mathcal{M}}
\newcommand\NN{\mathcal{N}}

\DeclareMathOperator{\supp}{supp}

\DeclareMathOperator{\SL}{SL}

\DeclareMathOperator{\SO}{SO}

\newcommand\N{\mathbb{N}}

\newcommand\R{\mathbb{R}}

\newcommand{\dd}{\mathop{}\!\mathrm{d}}

\DeclarePairedDelimiter\br{(}{)}

\DeclarePairedDelimiter\abs{\lvert}{\rvert}
\DeclarePairedDelimiter\norm{\lVert}{\rVert}

\providecommand\for{}
\newcommand\SetSymbol[1][]{%
\nonscript\:#1\vert
\allowbreak
\nonscript\:
\mathopen{}}
\DeclarePairedDelimiterX\set[1]{\lbrace}{\rbrace}{%
\renewcommand\for{\SetSymbol[\delimsize]}
#1
}

\newcommand{\ii}{\mathrm{i}}
\newcommand{\euler}{\mathrm{e}}

\renewcommand{\epsilon}{\varepsilon}

\renewcommand{\phi}{\varphi}

%% file: abstract.tex
\begin{abstract}
The purpose of this note is to establish convergence of random walks on the moduli space of Abelian differentials on compact Riemann surfaces in two different modes: 
convergence of the $n$-step distributions from almost every starting point in an affine invariant submanifold towards the associated affine invariant measure, and almost sure pathwise equidistribution towards the affine invariant measure on the $\SL_2(\R)$-orbit closure of an arbitrary starting point. 
These are analogues to previous results for random walks on homogeneous spaces. 
\end{abstract}

%% file: introduction.tex
\section{Introduction}\label{sec:intro}
Consider the moduli space of unit-area Abelian differentials on compact Riemann surfaces, that is the space of pairs $(M,\omega)$, where $M$ is a compact Riemann surface and $\omega$ a holomorphic 1-form on $M$ with $\operatorname{vol}(M,\omega)=\frac{\ii}{2}\int_M\omega\wedge\bar{\omega}=1$, up to biholomorphic equivalence. 
The form $\omega$ determines a flat metric on $M$ with conical singularities at its zeros. 
Hence, a pair $(M,\omega)$ can alternatively be seen as a \emph{translation surface}. 
This viewpoint can be used to define a natural $\SL_2(\R)$-action on the moduli space. 
The moduli space is stratified by specification of combinatorial data: 
the genus $g$ of the surface $M$ and a partition $\alpha=(\alpha_1,\dots,\alpha_n)$ of $2g-2$ giving the multiplicity of the zeros of $\omega$. 
Strata are not always connected but consist of at most three connected components, which have been classified \cite{conn_comp}. 
The $\SL_2(\R)$-action preserves strata and their connected components. 
We refer to the survey \cite{flat_survey} for further background. 
In the following, we restrict our attention to a connected component of a stratum, which we shall denote by $\H$ throughout the article. 

Inside $\H$ there are natural lower dimensional structures, called \emph{affine invariant submanifolds}, which are immersed submanifolds that locally look like complex subspaces in period coordinates; see \cite[Definition~1.2]{EMM}. 
Every affine invariant submanifold $\M$ carries a unique ergodic $\SL_2(\R)$-invariant probability measure $\nu_\M$. 
A particular case is $\H$ itself together with the normalized Masur--Veech measure. 
The following result will serve as our motivating example. 
\begin{theorem}[{Eskin--Mirzakhani--Mohammadi \cite{EMM}}]\label{thm:EMM}
Let $\mu$ be an absolutely continuous compactly supported $\SO_2(\R)$-bi-invariant probability measure on $\SL_2(\R)$ and $x\in\H$. 
Then the orbit closure $\overline{\SL_2(\R)x}$ is an affine invariant submanifold $\M$, and we have the weak* convergence 
\begin{align}\label{cesaro}
\frac1n\sum_{k=0}^{n-1}\mu^{*k}*\delta_x\longrightarrow \nu_\M
\end{align}
as $n\to\infty$.
\end{theorem}
Here $\mu^{*k}$ denotes the $k$-fold convolution power of $\mu$, and
weak* convergence of measures means convergence when the measures in question are applied to continuous test functions with compact support. 
Spelled out explicitly, the weak* convergence in the conclusion of the theorem above thus means that for every compactly supported continuous function $f\in C_c(\H)$ it holds that 
\begin{align*}
\lim_{n\to\infty}\frac1n\sum_{k=0}^{n-1}\int f(g_k\dotsm g_1x)\dd\mu^{\otimes k}(g_1,\dots,g_k)=\int f\dd\nu_\M.
\end{align*}
This result should be interpreted as a statement about Ces\`aro convergence in law of the random walk on $\H$ given by $\mu$. 
Indeed, the convolution $\mu^{*n}*\delta_x$ is the distribution of the location after $n$ steps of the random walk started at $x$. 
The purpose of this short article is to establish two further modes of convergence for such random walks:
\begin{enumerate}
\item \emph{Generic points for non-averaged convergence:} In \parsign\ref{sec:spectral_gap}, we prove that the stronger, non-averaged, weak* convergence
\begin{align*}
\mu^{*n}*\delta_x\longrightarrow \nu_\M
\end{align*}
as $n\to\infty$ holds for $\nu_\M$-almost every starting point $x$. 
As in the homogeneous setting (see \cite[\parsign 3]{aspects}), the key ingredient is the existence of a spectral gap in $L_0^2(\nu_\M)$ of the convolution operator 
\begin{align*}
\pi(\mu)\colon f\mapsto \biggl(x\mapsto \int f(gx)\dd\mu(g)\biggr)
\end{align*}
acting on measurable functions, which is a consequence of the work of Avila--Gou{\"e}zel~\cite{spectral_gap_SL2}. 
\item \emph{Pathwise equidistribution:} In \parsign\ref{sec:pathwise}, we improve the convergence in law in \eqref{cesaro} to almost sure pathwise convergence, meaning that for $\mu^{\otimes\N}$-almost every sequence $(g_i)_i$ of elements of $\SL_2(\R)$ we have 
\begin{align*}
\frac1n\sum_{k=0}^{n-1}\delta_{g_k\dotsm g_1x}\longrightarrow \nu_\M
\end{align*}
as $n\to\infty$ in the weak* topology. 
The argument uses techniques developed by Benoist--Quint for the homogeneous case \cite{BQ3} and the Lyapunov functions constructed in \cite[Proposition~2.13]{EMM}. 
\end{enumerate}

The removal of ergodic averages from convergence results is a recurring theme in current research on dynamical systems. 
In the context of Teichm\"uller dynamics, this idea features for example in the work of Nevo--R\"uhr--Weiss~\cite{eff-counting} as an ingredient to their proof of an effective counting estimate for saddle connection holonomy vectors, and Forni~\cite{forni} proved that the ergodic average can be removed for limits of geodesic push-forwards of horocycle-invariant measures outside a set of times of upper density zero. 

Regarding point~(ii) above, after completion of the first version of the present article the author was made aware that the problem of pathwise equidistribution of random walks on moduli space has already been studied by Chaika--Eskin~\cite[\S2]{chaika-eskin}. 
In~\cite[Theorem~2.1]{chaika-eskin}, they obtain the same conclusion using essentially the same techniques. 
The additional contribution to this problem in the present paper consists in a slight generalization of the class of measures $\mu$ to which the result applies; see the brief discussion after the statement of Theorem~\ref{thm:pathwise}. 
\subsection*{Acknowledgments}
The author would like to thank Jayadev Athreya for valuable discussions, his encouragement to write this article, and providing numerous useful references. 
Thanks are also due to Alex Eskin for pointing out further related results and the overlap with~\cite{chaika-eskin}. 

%% file: spectral_gap.tex
\section{Generic Points}\label{sec:spectral_gap}
Let $G$ be a locally compact $\sigma$-compact metrizable group and $X$ a locally compact $\sigma$-compact metrizable space on which $G$ acts continuously. 
Then for any probability measure $\mu$ on $G$, one can define the convolution operator $\pi(\mu)$ by
\begin{align*}
\pi(\mu)f(x)=\int f(gx)\dd\mu(g)
\end{align*}
for bounded measurable functions $f$ on $X$ and $x\in X$. 
If $m_X$ is a $G$-invariant probability measure on $X$, this gives a bounded linear operator $\pi(\mu)\colon L^\infty(m_X)\to L^\infty(m_X)$ which extends to a continuous contraction on each $L^p$-space (see \cite[Corollary~2.2]{BQ_book}). 

We will be interested in the existence of an $L^2$-spectral gap of this convolution operator in the case where $G=\SL_2(\R)$, $X=\M$ is an affine invariant submanifold of $\H$ endowed with the ergodic $\SL_2(\R)$-invariant probability measure $\nu_\M$, and $\mu$ is a probability measure on $\SL_2(\R)$. 
\begin{definition}
We say that $\mu$ has an \emph{$L^2$-spectral gap} on $X$ if the associated convolution operator $\pi(\mu)$ restricted to the space $L_0^2(X,m_X)$ of square-integrable functions with mean $0$ has spectral radius strictly less than $1$. 
\end{definition}
We note that by the spectral radius formula, $\mu$ having an $L^2$-spectral gap on $X$ can be reformulated as the requirement that 
\begin{align*}
\lim_{n\to\infty}\sqrt[n]{\norm{\pi(\mu)|^n_{L_0^2}}_{\mathrm{op}}}<1.
\end{align*}
\begin{proposition}\label{prop:spectral_gap}
Suppose that the probability measure $\mu$ on $\SL_2(\R)$ is not supported on a closed amenable subgroup and let $\M$ be an affine invariant submanifold of $\H$. 
Then $\mu$ has a $L^2$-spectral gap on $\M$. 
\end{proposition}
Underlying this proposition is seminal work on small eigenvalues of the foliated hyperbolic Laplacian on $\M$, due to Avila--Gou\"{e}zel--Yoccoz~\cite{AGY} in the case of strata and to Avila--Gou\"{e}zel~\cite{spectral_gap_SL2} in the general case. 
The proof of Proposition~\ref{prop:spectral_gap} mainly consists of a translation of these results into the language of representation theory, a connection that was already noted in~\cite[Appendix~B]{AGY}. 
For the reader's convenience, we briefly review some of the relevant concepts. 

First, recall that the $\SL_2(\R)$-action on $\M$ induces a unitary representation $\pi$ of $\SL_2(\R)$ on $L_0^2(\nu_\M)$ defined by $\pi_gf(x)=f(g^{-1}x)$ for $g\in\SL_2(\R)$, $f\in L_0^2(\nu_\M)$ and $x\in \M$. 
By ergodicity of $\nu_\M$, there are no nonzero $\SL_2(\R)$-invariant functions in $L_0^2(\nu_\M)$. 
As $\SL_2(\R)$ does not have property (T), there might however exist  functions that are \enquote{almost invariant}, which would represent an obstruction to the conclusion of the proposition. 
To introduce the latter concept, let $\pi$ be an arbitrary unitary representation of a locally compact $\sigma$-compact metrizable group $G$ on a separable Hilbert space $\H$. 
Then $\pi$ is said to have \emph{almost invariant vectors} if for every $\epsilon>0$ and every compact subset $Q\subset G$ there exists a unit vector $v\in \H$ such that $\norm{\pi_gv-v}<\epsilon$ for every $g\in Q$. 
Equivalently, the trivial representation $\mathds{1}$ of $G$ is \emph{weakly contained} in $\pi$, written $\mathds{1}\prec \pi$. 
On the set of equivalence classes of separable unitary representations of $G$ there is natural topology, called the \emph{Fell topology}. 
This topology is in general not Hausdorff and has the property that weak containment $\pi_1\prec \pi_2$ of two separable unitary representations $\pi_1,\pi_2$ of $G$ is equivalent to $\pi_1\in\overline{\set{\pi_2}}$. 
Thus, $\pi$ not having almost invariant vectors can be equivalently characterized as $\pi$ being isolated from the trivial representation in the Fell topology, in the sense that $\mathds{1}\notin\overline{\set{\pi}}$. 
We refer to \cite[Appendix~F]{BHV} and the references therein for further background on these notions. 
\begin{proof}[Proof of Proposition~\textup{\ref{prop:spectral_gap}}]
The convolution operator $\pi(\mu)$ associated to $\mu$ can be expressed as $\pi(\mu)f = \int\pi_{g^{-1}}f\dd\mu(g)$, where $\pi$ denotes the unitary representation of $\SL_2(\R)$ on $L_0^2(\nu_\M)$ described above and the integral may either be understood pointwise or as weak integral. 
In view of a theorem of Shalom (\cite[Theorem~C]{spectral_gap} applied to the push-forward of $\mu$ by the inverse map $g\mapsto g^{-1}$), it is enough to show that the trivial representation is not weakly contained in $\pi$, or in other words, that $\pi$ is isolated from the trivial representation in the Fell topology. 

To this end, we need to draw on the representation theory of $\SL_2(\R)$, the relevant parts of which are summarized in \cite[\S3.4]{spectral_gap_SL2}. 
The crucial fact is that the only non-trivial irreducible unitary representations of $\SL_2(\R)$ converging to the trivial representation are complementary series representations $\mathcal{C}^u$ with parameter $u\in(0,1)$ approaching~$1$. 
Recalling that the trivial representation is not contained in $\pi$ by ergodicity and using properties of the Fell topology, it follows that $\pi$ is isolated from the trivial representation if and only if there is an upper bound $\overline{u}<1$ such that the representations $\mathcal{C}^u$ for $u>\overline{u}$ do not feature with positive weight in the direct integral decomposition of $\pi$ into irreducibles. 
Whether the latter holds can be understood by considering the spectrum on $L_0^2(\nu_\M)$ of the Casimir operator $\Omega$ of $\SL_2(\R)$, which is a differential operator generating the center of the universal enveloping algebra of $\SL_2(\R)$. 
Being central, $\Omega$ acts as scalar in irreducible unitary representations of $\SL_2(\R)$, and this scalar equals $(1-u^2)/4$ for the complementary series representation $\mathcal{C}^u$. 
The desired property of $\mathcal{C}^u$ not featuring in the integral decomposition of $\pi$ for values of $u$ arbitrarily close to $1$ is thus equivalent to $0$ not being an accumulation point of the spectrum $\sigma(\Omega)$ of $\Omega$ on $L_0^2(\nu_\M)$. 
As explained at the end of~\cite[\S3.4]{spectral_gap_SL2}, the Casimir operator can be interpreted as foliated Laplacian $\Delta$ on $\M$, which implies that $\sigma(\Omega)\cap(0,1/4)=\sigma(\Delta)\cap(0,1/4)$. 
However, as a consequence of the Main Theorem of~\cite{spectral_gap_SL2}, zero is not an accumulation point of the spectrum of $\Delta$ on $\M$. 
This entails the same property for $\Omega$ on $L_0^2(\nu_\M)$, hence the desired statement about complementary series representations featuring in $\pi$.
This finally proves that $\pi$ is indeed isolated from the trivial representation and finishes the proof. 
\end{proof}
We are now ready to state and prove the following quantitative result on generic points for random walk convergence. 
\begin{theorem}\label{thm:ae_convergence}
Let $\H$ be a connected component of a stratum of the moduli space of unit-area Abelian differentials on compact Riemann surfaces and $\M$ an affine invariant submanifold carrying the ergodic $\SL_2(\R)$-invariant measure $\nu_\M$. 
Let $\mu$ be a probability measure on $\SL_2(\R)$ that is not supported on a closed amenable subgroup. 
Then for $\nu_\M$-almost every $x\in \M$ we have 
\begin{align}\label{qualitative}
\mu^{*n}*\delta_x\longrightarrow \nu_\M
\end{align}
as $n\to\infty$ in the weak* topology. 
This convergence is exponentially fast in the sense that for every fixed $f\in L^2(\nu_\M)$ we have 
\begin{align}\label{exponential}
\limsup_{n\to\infty}\abs*{\int f\dd(\mu^{*n}*\delta_x)-\int f\dd \nu_\M}^{1/n}\le \rho\bigl(\pi(\mu)|_{L_0^2}\bigr)^{1/2}
\end{align}
for $\nu_\M$-a.e.\ $x\in \M$, where $\rho\bigl(\pi(\mu)|_{L_0^2}\bigr)$ denotes the spectral radius of $\pi(\mu)$ restricted to $L_0^2(\nu_\M)$. More precisely, given $\rho\bigl(\pi(\mu)|_{L_0^2}\bigr)<\alpha<1$, choose $N\in\N$ such that $\norm{\pi(\mu)|^n_{L_0^2}}_{\mathrm{op}}\le \alpha^n$ for all $n\ge N$. 
Then if we denote $f_0=f-\int f\dd\nu_\M$ and
\begin{align*}
B_{\alpha,n,f}=\set*{x\in \M\for \abs[\Big]{\pi(\mu)^{n'}f(x)-\int f\dd \nu_\M}\ge \alpha^{n'/2}\norm{f_0}_{L^2}\text{ for some }n'\ge n},
\end{align*}
we have the bound
\begin{align}\label{quantitative}
\nu_\M\br*{B_{\alpha,n,f}}\le \frac{\alpha^n}{1-\alpha}
\end{align}
for every $n\ge N$. 
\end{theorem}
To make sense of this statement, recall that $\rho\bigl(\pi(\mu)|_{L_0^2}\bigr)$ is guaranteed to be strictly less than $1$ by Proposition~\ref{prop:spectral_gap}. 
Examples of measures to which the theorem applies are Zariski dense measures (i.e.\ measures whose support generates a Zariski dense subgroup of $\SL_2(\R)$) and also the measures appearing in Theorem~\ref{thm:EMM}. 
\begin{proof}
In view of separability of $C_c(\H)$, the weak* convergence \eqref{qualitative} will follow if we can prove that for a fixed function $f\in C_c(\H)$ we have
\begin{align*}
\pi(\mu^{*n})f=\pi(\mu)^nf\longrightarrow\int f\dd \nu_\M
\end{align*}
$\nu_\M$-a.e.\ as $n\to\infty$. 
Since $\rho\bigl(\pi(\mu)|_{L_0^2}\bigr)<1$ holds by Proposition~\ref{prop:spectral_gap}, this follows from \eqref{exponential}. 
Validity of \eqref{exponential} in turn follows from \eqref{quantitative} by an application of Borel--Cantelli and after letting $\alpha$ approach $\rho\bigl(\pi(\mu)|_{L_0^2}\bigr)$. 

Thus, it suffices to establish \eqref{quantitative}. 
To this end, observe first that, for every $n\ge N$, 
\begin{align*}
\norm*{\pi(\mu)^nf-\int f\dd \nu_\M}_{L^2}=\norm{\pi(\mu)^nf_0}_{L^2}\le \norm{\pi(\mu)|^n_{L_0^2}}_{\mathrm{op}}\norm{f_0}_{L^2}\le \alpha^n\norm{f_0}_{L^2}.
\end{align*}
By Chebyshev's inequality, it follows that for $n\ge N$ we have 
\begin{align*}
\nu_\M\br*{\set*{x\in X\for \abs*{\pi(\mu)^nf(x)-\int f\dd \nu_\M}\ge \alpha^{n/2}\norm{f_0}_{L^2}}}&\le \frac{\norm*{\pi(\mu)^nf-\int f\dd \nu_\M}_{L^2}^2}{\alpha^n\norm{f_0}_{L^2}^2}\\
&\le\alpha^n.
\end{align*}
Summing over $n'\ge n$ gives the bound \eqref{quantitative}. 
\end{proof}

%% file: pathwise.tex
\section{Pathwise Equidistribution}\label{sec:pathwise}
In this section, we aim to prove the following theorem. 
For the statement, recall that a measure $\mu$ on a linear group $G$ is said to have \emph{finite exponential moments} if for $\delta>0$ small enough, the function $g\mapsto \norm{g}^\delta$ is $\mu$-integrable, where $\norm{\cdot}$ denotes any matrix norm. 
\begin{theorem}\label{thm:pathwise}
Let $\H$ be a connected component of a stratum of the moduli space of unit-area Abelian differentials on compact Riemann surfaces, let $x\in\H$ and $\M=\overline{\SL_2(\R)x}$ be the minimal affine invariant submanifold containing $x$ endowed with its ergodic $\SL_2(\R)$-invariant measure $\nu_\M$. 
Moreover, let $\mu$ be an $\SO_2(\R)$-right-invariant probability measure on $\SL_2(\R)$ with finite exponential moments satisfying $\mu(\SO_2(\R))=0$. 
Then for $\mu^{\otimes\N}$-a.e.\ sequence $(g_i)_i$ we have 
\begin{align*}
\frac1n\sum_{k=0}^{n-1}\delta_{g_k\dotsm g_1x}\longrightarrow \nu_\M
\end{align*}
as $n\to\infty$ in the weak* topology. 
\end{theorem}
The above theorem is essentially Chaika--Eskin's~\cite[Theorem~2.1]{chaika-eskin}, albeit with slightly weaker assumptions on the measure $\mu$. 
Indeed, in \cite{chaika-eskin} $\mu$ is assumed to be $\SO_2(\R)$-bi-invariant, absolutely continuous with respect to Haar measure on $\SL_2(\R)$, and to be compactly supported. 
While the former two differences in assumptions are largely insignificant, the upgrade from compact support to finite exponential moments does require a bit of care. 

The first step towards the proof of results such as Theorem \ref{thm:EMM} or Theorem \ref{thm:pathwise} always is the classification of (ergodic) $\mu$-stationary measures. 

Recall that given a continuous group action of $G$ on $X$ (as at the beginning of \S\ref{sec:spectral_gap}), a probability measure $\nu$ on $X$ is called \emph{$\mu$-stationary} if $\mu*\nu=\nu$, which means in more detail that
\begin{align*}
\int \int f(gx)\dd\mu(g)\dd\nu(x)=\int f\dd\nu
\end{align*}
for every bounded measurable function $f$ on $X$.

In \cite{EMi}, Eskin--Mirzakhani prove that $\mu$-stationary measures are necessarily affine when $\mu$ is absolutely continuous and $\SO_2(\R)$-bi-invariant. 
In fact, what they prove is that ergodic $P$-invariant measures are affine, where $P\subset \SL_2(\R)$ denotes the upper triangular subgroup, and then use that $\mu$-stationary measures are in correspondence with $P$-invariant measures by classical results of Furstenberg \cite{F1,F2} (see also \cite[Theorem~1.4]{NZ} for a concise restatement). 
These results of Furstenberg apply whenever $\mu$ is \emph{admissible}, meaning that $\supp(\mu)$ generates $\SL_2(\R)$ as a semigroup and some convolution power $\mu^{*k}$ is absolutely continuous with respect to Haar measure on $\SL_2(\R)$. 
We can therefore record the following. 
\begin{theorem}[\cite{EMi}]\label{thm:stationary}
Suppose that $\mu$ is admissible in the sense above. 
Then any ergodic $\mu$-stationary probability measure on $\H$ is affine. 
\end{theorem}
Let us quickly convince ourselves that this result applies to the measures in Theorem~\ref{thm:pathwise}. 
\begin{corollary}\label{cor:stationary}
If $\mu$ is $\SO_2(\R)$-right-invariant and satisfies $\mu(\SO_2(\R))=0$, then any ergodic $\mu$-stationary probability measure on $\H$ is affine. 
\end{corollary}
\begin{proof}
We consider the $KAK$-decomposition of $\SL_2(\R)$, where $K=\SO_2(\R)$ and
\begin{align*}
A=\set*{a_t\coloneqq\begin{pmatrix}\euler^t&0\\0&\euler^{-t}\end{pmatrix}\for t\in\R}.
\end{align*}
Then, in view of $K$-right-invariance of $\mu$, we can decompose $\mu$ as
\begin{align}\label{mu_representation}
\mu=\int\mu_{K,t}*\delta_{a_t}*m_K\dd\eta(t),
\end{align}
where $\eta$ is a probability measure on $\R$, $\mu_{K,t}$ is a probability measure on $K$ for every $t\in\R$, $\delta_{a_t}$ denotes the Dirac mass at $a_t\in A$, and $m_K$ is the Haar probability measure on $K$. 
The assumption that $\mu(K)=0$ then translates to the statement that $\eta(\set{0})=0$. 
It is not difficult to see, e.g.\ by hyperbolic geometry considerations, that $m_K*a_s*m_K*a_t*m_K$ is absolutely continuous with respect to Haar measure on $\SL_2(\R)$ whenever $s,t\neq 0$. 
Calculating the third convolution power using the representation \eqref{mu_representation} of $\mu$ above, it follows that $\mu^{*3}$ is absolutely continuous as well. 
Also by hyperbolic geometry, $\supp(\mu)$ generates $\SL_2(\R)$ as a semigroup. 
Hence, $\mu$ is admissible and Theorem~\ref{thm:stationary} applies. 
\end{proof}

Even though there is more to be done, let us already now give the proof of the main theorem of this section, as it will motivate the remaining work.
\begin{proof}[Proof of Theorem \textup{\ref{thm:pathwise}}]
We consider the one-point compactification $\overline{\M}=\M\cup\set{\infty}$ of the affine invariant submanifold $\M$ and are going to prove the desired convergence statement inside the space $\mathcal{P}(\overline{\M})$ of probability measures on $\overline{\M}$. 
As $\mathcal{P}(\overline{\M})$ is compact in the weak* topology, it suffices to show that, almost surely, every weak* limit in $\mathcal{P}(\overline{\M})$ of the sequence $(\frac1n\sum_{k=0}^{n-1}\delta_{g_k\dotsm g_1x})_n$ of empirical measures equals $\nu_\M$. 

To this end, we start by observing that by the Breiman law of large numbers (see \cite[Corollary~3.3]{BQ3}), for $\mu^{\otimes\N}$-a.e.\ sequence $(g_i)_i$ every weak* limit $\nu$ of the sequence of empirical measures is $\mu$-stationary. 
Let $F_1$ be a full measure set with respect to $\mu^{\otimes\N}$ such that this conclusion holds for every $(g_i)_i\in F_1$. 
By Corollary~\ref{cor:stationary}, for $(g_i)_i\in F_1$ the measures featuring with positive weight in the ergodic decomposition of a weak* limit $\nu$ as above can only be affine measures $\nu_\NN$ with $\NN\subset\M$ and the point mass~$\delta_\infty$ at infinity, the latter corresponding to potential escape of mass. 
We will show below (Corollary~\ref{cor:BQ} and Proposition~\ref{prop:lyapunov}) that $\delta_\infty$ and any given affine measure $\nu_\NN$ with $\NN\subsetneq\M$ do not appear in the decomposition with positive weight for $\mu^{\otimes\N}$-a.e.\ $(g_i)_i$. 
Let $F\subset F_1$ be the intersection of these full measure sets associated to $\delta_\infty$ and all $\nu_\NN$ for $\NN\subsetneq\M$. 
As $\M$ admits only countably many proper affine invariant submanifolds (see \cite[Proposition~2.16]{EMM}), $F$ still has full measure. 
Choosing $(g_i)_i$ in $F$, we conclude that almost every ergodic component of every weak* limit $\nu$ of the sequence of empirical measures equals $\nu_\M$, which proves the desired pathwise convergence statement. 
\end{proof}
We see that it remains to rule out the occurrence of unwanted limit measures. The key tool to achieve this is the following concept. 
\begin{definition}\label{def:lyapunov}
Consider a measurable group action of $G$ on a standard Borel space $X$. 
A measurable function $V\colon X\to[0,\infty]$ is called a \emph{Lyapunov function} for the random walk on $X$ induced by a probability measure $\mu$ on $G$ if there exist constants $\alpha\in(0,1)$, $\beta\ge 0$ such that $\pi(\mu)V\le \alpha V+\beta$, where $\pi(\mu)$ is the associated convolution operator introduced in \parsign\ref{sec:spectral_gap}. 
\end{definition}

Intuitively speaking, the contraction inequality means that after a step of the random walk, the value of the Lyapunov function $V$ on average gets smaller by a constant factor, at least outside some compact set where the value of $V$ lies below some threshold depending on the additive constant $\beta$. 
The dynamics are therefore directed towards the part of the space where $V$ takes small values. 
The following quantification of this phenomenon is due to Benoist--Quint, but similar ideas already have a long and successful tradition in Markov chain theory (see e.g.~\cite[Theorem~18.5.2]{meyn-tweedie} and the references given there). 
\begin{proposition}[{\cite[Proposition~3.9]{BQ3}}]\label{prop:BQ}
Suppose the random walk on $X$ induced by $\mu$ admits a Lyapunov function $V\colon X\to[0,\infty]$. 
Then there exists a constant $C>0$ such that for any $x\in X$ with $V(x)<\infty$, for $\mu^{\otimes\N}$-a.e.\ $(g_i)_i$ we have for any $M>0$
\begin{align*}
\limsup_{n\to\infty}\frac1n\abs{\set{0\le k<n\for V(g_k\dotsm g_1x)>M}}\le \frac{C}{M}.
\end{align*}
\end{proposition}
\begin{corollary}\label{cor:BQ}
Let $G$ be a locally compact $\sigma$-compact metrizable group and $X$ a locally compact $\sigma$-compact metrizable space endowed with a continuous $G$-action. 
Let $\mu$ be a probability measure on $G$ and suppose that the induced random walk on $X$ admits a Lyapunov function $V\colon X\to[0,\infty]$ with the additional property that for every $M>0$ the sublevel set $X_M=V^{-1}([0,M])$ is relatively compact and its closure $\overline{X_M}$ is contained in $X\setminus V^{-1}(\set{\infty})$. 
Then for $\mu^{\otimes\N}$-a.e.\ $(g_i)_i$, any weak* limit $\nu$ of $(\frac1n\sum_{k=0}^{n-1}\delta_{g_k\dotsm g_1x})_n$ satisfies $\nu(X\setminus V^{-1}(\set{\infty}))=1$. 
\end{corollary}
\begin{proof}
Fix a sequence $(g_i)_i$ such that the conclusion of Proposition~\ref{prop:BQ} holds. 
Since the measure $\nu$ is regular, the assumptions imply that for every $\epsilon>0$ and $M>0$ there exists a non-negative compactly supported continuous function $f_{\epsilon,M}$ on $X$ bounded by $1$ which takes the value $1$ on $\overline{X_M}$ such that
\begin{align*}
\nu(\overline{X_M})\ge \int f_{\epsilon,M}\dd\nu-\epsilon.
\end{align*}
Applying weak* convergence to this function, it follows that
\begin{align*}
\nu(X\setminus V^{-1}(\set{\infty}))\ge\nu(\overline{X_M})\ge\int f_{\epsilon,M}\dd\nu-\epsilon\ge 1-C/M-\epsilon. 
\end{align*}
Letting $M\to\infty$ and $\epsilon\to 0$ thus establishes the claim. 
\end{proof}
In view of the above, it remains to find Lyapunov functions $V$ on $\H$ taking the value $\infty$ precisely on a given affine invariant submanifold $\NN$ and satisfying the properness conditions in Corollary~\ref{cor:BQ}. 
The case $\NN=\emptyset$ (responsible for ruling out escape of mass, i.e.\ the occurrence of $\delta_\infty$ as part of the limit measure) was dealt with by Athreya \cite{Ath}; most of the work necessary for the general case was carried out in \cite{EMM}.
\begin{proposition}[{\cite[Proposition~2.13]{EMM}}]\label{prop:lyapunov_pre}
For $t>0$, define $\mu_t= (a_t)_*m_K$. 
Let $\NN\subset \H$ be an affine invariant submanifold \textup{(}$\NN=\emptyset$ is allowed\textup{)}. 
Then there exists $\beta\ge 0$ and an $\SO_2(\R)$-invariant function $f_\NN\colon\H\to[1,\infty]$ with the following properties: 
\begin{enumerate}
\item $f_\NN^{-1}(\set{\infty})=\NN$ and for every $M>0$ the closure of the sublevel set $f_\NN^{-1}([0,M])$ is compact and contained in $\H\setminus\NN$,
\item for every $0<\alpha<1$ there exists $t_0$ such that for $t\ge t_0$ it holds that
\begin{align*}
\pi(\mu_t)f_\NN\le \alpha f_\NN+\beta,
\end{align*}
and
\item for some $\sigma>1$ and all $g$ in a neighborhood of the identity in $\SL_2(\R)$ we have
\begin{align*}
\sigma^{-1}f_\NN(x)\le f_\NN(gx)\le \sigma f_\NN(x)
\end{align*}
for all $x\in \H$.
\end{enumerate}
\end{proposition}
Since any $g\in\SL_2(\R)$ is a product of at most $O(\log\norm{g})+1$ elements of a given neighborhood of the identity, iterating (iii) above we more generally obtain:
\begin{enumerate}
\item[(iii')] there exist constants $\sigma>1,\kappa>0$ such that for every $g\in\SL_2(\R)$ and $x\in\H$
\begin{align*}
\sigma^{-1}\norm{g}^{-\kappa}f_\NN(x)\le f_\NN(gx)\le \sigma\norm{g}^\kappa f_\NN(x).
\end{align*}
\end{enumerate}

The final step is to use the functions provided by the above proposition to construct the required Lyapunov functions for the measures from the statement of Theorem \ref{thm:pathwise}.
\begin{proposition}\label{prop:lyapunov}
Let $\mu$ be a probability measure on $\SL_2(\R)$ with finite exponential moments that is $\SO_2(\R)$-right-invariant and not equal to the Haar measure on $\SO_2(\R)$. 
Let $\NN\subset \H$ be an affine invariant submanifold. 
Then there exists a Lyapunov function $V_\NN$ for $\mu$ with $V_\NN^{-1}(\set{\infty})=\NN$ that satisfies the conditions of Corollary~\textup{\ref{cor:BQ}}. 
\end{proposition}
The proof is an extension of the argument for \cite[Lemma~3.2]{EMM}. 
\begin{proof}
We will first show that there exists $m\in\N$ such that the function $f_\NN$ provided by Proposition~\ref{prop:lyapunov_pre} is a Lyapunov function for the $m$-step random walk, i.e.\ with $\pi(\mu)^mf_\NN\le \alpha f_\NN+\beta$ for some $\alpha<1,\beta\ge 0$. 

We first treat the case of a $K\coloneqq \SO_2(\R)$-bi-invariant measure $\mu$. 
In this case, also the convolution powers of $\mu$ are $K$-bi-invariant, so that for every $m\in\N$ there exists a probability measure $\smash{\eta^{(m)}}$ on $\R_+$ such that 
\begin{align}\label{mum_decomp}
\mu^{*m}=\int_{\R_+}m_K*\delta_{a_t}*m_K\dd\eta^{(m)}(t).
\end{align}
Since the random walk on $\SL_2(\R)$ given by $\mu$ has a positive top Lyapunov exponent by Furstenberg's theorem~\cite{F1} (see~\cite[Theorem~II.4.1]{bl} for a precise restatement in the context at hand), we know that $\norm{Y_m\dotsm Y_1}$ almost surely grows exponentially, where $(Y_i)_i$ is a sequence of i.i.d.\ random matrices with common distribution $\mu$. 
As the quantity $\smash{\eta^{(m)}}([0,t_0])$ for a fixed $t_0>0$ represents the probability that $Y_m\dotsm Y_1$ lies in some fixed bounded subset of $\SL_2(\R)$, it follows that $\smash{\eta^{(m)}}([0,t_0])\to 0$ as $m\to\infty$. 
Next, from property~(ii) in Proposition~\ref{prop:lyapunov_pre} we know that given $\alpha\in(0,1)$ there exists $t_0$ such that $\pi(\mu_t)f_\NN\le \frac{\alpha}{2}f_\NN+\beta$ for all $t\ge t_0$, where $\mu_t=(a_t)_*m_K$ is the measure defined in that proposition. 
Iterating property~(iii) of the function~$f_\NN$, there exists some constant $R>0$ such that $f_\NN(a_tkx)\le Rf_\NN(x)$ for all $k\in K$, $0\le t\le t_0$ and $x\in\H$. 
Using \eqref{mum_decomp} and $K$-invariance of $f_\NN$, we thus find
\begin{align*}
\pi(\mu)^mf_\NN(x)&=\int_K\int_0^\infty\int_K f_\NN(k'a_tkx)\dd m_K(k')\dd\eta^{(m)}(t)\dd m_K(k)\\
&=\int_K\int_0^\infty f_\NN(a_tkx)\dd\eta^{(m)}(t)\dd m_K(k)\\
&=\int_0^\infty \pi(\mu_t)f_\NN(x)\dd\eta^{(m)}(t)\\
&=\int_0^{t_0} \pi(\mu_t)f_\NN(x)\dd\eta^{(m)}(t)+\int_{t_0}^\infty \pi(\mu_t)f_\NN(x)\dd\eta^{(m)}(t)\\
&\le R\eta^{(m)}([0,t_0])f_\NN(x)+\frac{\alpha}{2}f_\NN(x)+\beta
\end{align*} 
for all $x\in \H$. 
Since as noted before, the term $R\smash{\eta^{(m)}}([0,t_0])$ tends to $0$ as $m\to\infty$, the right-hand side above is bounded by $\alpha f_\NN(x)+\beta$ for sufficiently large $m$, which is what we needed. 

The argument for merely $K$-right-invariant $\mu$ can be reduced to the case above. 
Indeed, if we set $\tilde{\mu}= m_K*\mu$, then what we have already established implies $\pi(\tilde{\mu})^mf_\NN\le \alpha f_\NN+\beta$ for all large $m$, and using again $K$-invariance of $f_\NN$ we see
\begin{align*}
\pi(\mu)^mf_\NN(x)=\int f_\NN(gx)\dd\mu^{*m}(g)=\int f_\NN(gx)\dd(m_K*\mu^{*m})(g)=\pi(\tilde{\mu})^mf_\NN(x)
\end{align*}
for all $x\in\H$, since $m_K*\mu^{*m}=\tilde{\mu}^{*m}$. 

Finally, let $\kappa$ be the constant from property~(iii') after Proposition~\ref{prop:lyapunov_pre} and choose $\delta\in(0,1)$ such that $\int\norm{g}^{\kappa\delta}\dd\mu(g)<\infty$, where we are using that $\mu$ has finite exponential moments. 
We define
\begin{align*}
V_\NN=\sum_{k=0}^{m-1}\alpha^{\frac{\delta(m-1-k)}{m}}\pi(\mu)^kf_\NN^\delta
\end{align*}
and claim that $V_\NN$ satisfies a contraction property with respect to the measure $\mu$. 
To see this, note first that $\pi(\mu)^mf_\NN^\delta\le (\alpha f_\NN+\beta)^\delta\le \alpha^\delta f_\NN^\delta+\beta^\delta$, in view of Jensen's inequality and the fact that $(s+t)^\delta\le s^\delta+t^\delta$ for non-negative real numbers $s,t$ and $\delta\in(0,1)$. 
Using this, one calculates
\begin{align*}
\pi(\mu)V_\NN &= \pi(\mu)^mf_\NN^\delta+\sum_{k=0}^{m-2}\alpha^{\frac{\delta(m-1-k)}{m}}\pi(\mu)^{k+1}f_\NN^\delta\\
&\le \alpha^{\delta/m}\biggl(\alpha^{\frac{\delta(m-1)}{m}}f_\NN^\delta+\sum_{k=0}^{m-2}\alpha^{\frac{\delta(m-1-(k+1))}{m}}\pi(\mu)^{k+1}f_\NN^\delta\biggr)+\beta^\delta\\
&=\alpha^{\delta/m}V_\NN+\beta^\delta,
\end{align*}
hence the claim. 
Moreover, using the first summand in the definition of $V_\NN$ as lower bound yields the inclusion $V_\NN^{-1}([0,M])\subset f_\NN^{-1}([0,\alpha^{(1-m)/m}M^{1/\delta}])$, so that the closure of the sublevel set $V_\NN^{-1}([0,M])$ is compact and contained in $\H\setminus\NN$ due to the corresponding property of $f_\NN$. 
It remains to argue that $V_\NN^{-1}(\set{\infty})=\NN$. 
The inclusion \enquote{$\supset$} is clear. 
To see that also the reverse inclusion holds, we use property~(iii') of the function $f_\NN$ to obtain that
\begin{align*}
\pi(\mu)^kf_\NN^\delta(x)\le \sigma^\delta\int\norm{g}^{\kappa\delta}\dd\mu^{*k}(g)f_\NN^\delta(x)\le\sigma^\delta\biggl(\int\norm{g}^{\kappa\delta}\dd\mu(g)\biggr)^kf_\NN^\delta(x)<\infty
\end{align*}
for $0\le k<m$ and any $x\in \H\setminus\NN$, by choice of $\delta$ and since $f_\NN^{-1}(\set{\infty})=\NN$. 
Hence, for $x\in\H\setminus\NN$ we also have $V_\NN(x)<\infty$. 
This finishes the proof. 
\end{proof}